\documentclass{amsart}
\usepackage{adjustbox,lipsum}
\usepackage{amsmath}
\usepackage{amssymb}
\usepackage{MnSymbol}
\usepackage[shortlabels]{enumitem}
\usepackage{tabu}
\usepackage{tikz-cd}
\usepackage{extarrows}
\usepackage{makecell}

\usepackage[
        colorlinks, citecolor=darkgreen,
        backref,
        pdfauthor={Samir Siksek},
]{hyperref}

\usepackage{comment}

\newcommand{\Aff}{\mathbb{A}}

\newcommand{\C}{\mathbb{C}}

\newcommand{\KK}{\mathbb{K}}
\newcommand{\PP}{\mathbb{P}}
\newcommand{\Q}{\mathbb{Q}}

\newcommand{\cP}{\mathcal{P}}

\newcommand{\cZ}{\mathcal{Z}}

\newcommand{\fm}{\mathfrak{m}}

\newcommand{\OO}{\mathcal{O}}

\DeclareMathOperator{\PSL}{PSL}
\DeclareMathOperator{\Stab}{Stab}
\DeclareMathOperator{\Fix}{Fix}
\DeclareMathOperator{\fpr}{fpr}
\DeclareMathOperator{\ind}{ind}

\DeclareMathOperator{\Orb}{Orb}
\DeclareMathOperator{\genus}{genus}

\DeclareMathOperator{\Aut}{Aut}

\DeclareMathOperator{\divv}{div}

\DeclareMathOperator{\Gal}{Gal}

\DeclareMathOperator{\BV}{BV}

\DeclareMathOperator{\Ram}{Ram}

\DeclareMathOperator{\ord}{ord}

\DeclareMathOperator{\prim}{prim}

\DeclareMathOperator{\Sym}{Sym}

\renewcommand{\setminus}{-}

\newtheorem{thm}{Theorem}

\newtheorem{lem}[thm]{Lemma}

\newtheorem{cor}[thm]{Corollary}

\theoremstyle{definition}

\theoremstyle{remark}

\definecolor{darkgreen}{rgb}{0,0.5,0}

\begin{document}

\title[]{
A single source theorem for\\ 
primitive points on curves
}

\begin{abstract}
Let $C$ be a curve defined over a number field $K$ and write $g$ for the
genus of $C$ and $J$ for the Jacobian of $C$.
Let $n \ge 2$. We say that an algebraic point $P \in C(\overline{K})$
has degree $n$ if the extension $K(P)/K$ has degree $n$.
By the Galois group of $P$ we mean the Galois group of
the Galois closure of $K(P)/K$ which
we identify as a transitive subgroup of $S_n$. We say that $P$
is primitive if its Galois group is primitive as a subgroup of $S_n$.
We prove the following \lq single source\rq\ theorem
for primitive points. Suppose $g>(n-1)^2$ if $n \ge 3$
and $g \ge 3$ if $n=2$.
Suppose that either $J$ is simple,
or that $J(K)$ is finite. Suppose $C$ has infinitely
many primitive degree $n$ points. Then
there is a degree $n$ morphism $\varphi : C \rightarrow \PP^1$
such that all but finitely many primitive degree $n$
points correspond to fibres $\varphi^{-1}(\alpha)$
with $\alpha \in \PP^1(K)$.

We prove moreover, under the same hypotheses, that if $C$ has infinitely many
degree $n$ points with Galois group $S_n$
or $A_n$, then $C$ has only finitely many degree $n$ points
of any other primitive Galois group. The proof makes
essential use of recent results of Burness and Guralnick
on fixed point ratios of faithful, primitive group actions.
\end{abstract}

\author{Maleeha Khawaja}

\address{
	School of Mathematics and Statistics\\
	Hicks Building\\
	University of Sheffield\\
	Sheffield S3 7RH\\
	United Kingdom
	}
\email{mkhawaja2@sheffield.ac.uk}

\author{Samir Siksek}

\address{Mathematics Institute\\
    University of Warwick\\
    CV4 7AL \\
    United Kingdom}

\email{s.siksek@warwick.ac.uk}

\date{\today}
\thanks{
Khawaja is supported by an EPSRC studentship from the University of Sheffield (EP/T517835/1).
Siksek is supported by the
EPSRC grant \emph{Moduli of Elliptic curves and Classical Diophantine Problems}
(EP/S031537/1). }

\makeatletter
\@namedef{subjclassname@2020}{%
  \textup{2020} Mathematics Subject Classification}
\makeatother


\maketitle

\section{Introduction}

Low degree points on curves have been a subject of intensive study.
Perhaps the most celebrated result in this subject
is Merel's uniform boundedness theorem \cite{Merel} which
asserts that the only degree $n$ points on the modular
curve $X_1(p)$ (with $p$ prime) are cuspidal, for $n< 2 \log_{3}(\sqrt{p}-1)$.
A common theme in the subject is to seek a description of which
curves can have infinitely many points of a certain degree.
For example, a famous theorem of Harris and Silverman
\cite{HS} asserts that if a curve $C/\Q$ of genus $\ge 2$
has infinitely many quadratic points then it is either hyperelliptic
or bielliptic. The strongest results to date on low degree points
on curves are due to Smith and Vogt \cite{SmithVogt} who prove
several theorems relating the minimal $n$ for which $C$ has
infinitely many degree $n$ points to the $K$-gonality of $C$.
By comparison, the question of which groups arise infinitely often as Galois groups 
of low degree points on a curve has received very little attention. 
This paper is concerned with giving insights into this question for 
primitive groups. Before we go further we recall the notion of
a primitive permutation group.
Let $G$ be a group acting on a finite set $\Omega$.
We say the action is \textbf{primitive} if
it is transitive and
the only partitions of $\Omega$ that are
$G$-stable are $\{\Omega\}$
and $\{\{\omega\} : \omega \in \Omega\}$.
It is well-known that a $2$-transitive group
acts primitively (Lemma~\ref{lem:2trans} below), and thus $S_n$ and $A_n$
are primitive groups (with their natural
action on $\{1,2,\dotsc,n\}$), for $n \ge 1$
and $n \ge 3$ respectively.

\medskip

Let $K$ be a perfect field 
and let $\overline{K}$ denote a fixed algebraic closure of $K$.
Write $G_K=\Gal(\overline{K}/K)$ for the absolute Galois group of $K$.
Let $C$ be a curve defined over $K$
(by which we mean a smooth projective
and geometrically irreducible variety defined over $K$ having 
dimension $1$). By a degree $n$ point on $C/K$ we mean an algebraic
point $P \in C(\overline{K})$ such that $[K(P):K]=n$. 
Equivalently, the orbit of $P$
under the action of $G_K$ has size $n$. 
If the orbit of $P$ is $\{P_1,\dotsc,P_n\}$,
then we define the \textbf{Galois group of $P$}, which we denote by
$\Gal(P/K)$, 
to be the image of the natural permutation
representation $G_K \rightarrow \Sym(\{P_1,\dotsc,P_n\})$.
Thus we may identify $\Gal(P/K)$ (up to conjugation) as a transitive
subgroup of the $n$-th symmetric group $S_n$.
The Galois group of $P$ is also the Galois group
of the Galois closure of $K(P)/K$.
We say that the point $P$ is \textbf{primitive} if $\Gal(P/K)$ acts
primitively on $\{P_1,\dotsc,P_d\}$. 
We call a divisor $D$ on $C$ \textbf{rational} if it is supported
on $C(\overline{K})$ and stable under the action of $G_K$.
Henceforth all divisors considered are assumed to be rational.
An effective divisor $D$ is said to be \textbf{reducible} if it admits a decomposition
$D=D_1+D_2$ where $D_1>0$, $D_2>0$ and both are rational, otherwise we say
that $D$ is irreducible. Thus an irreducible divisor consists of 
a single Galois orbit of algebraic points. We call an irreducible
divisor \textbf{primitive} if it is the Galois orbit of a primitive point.

\medskip

\begin{thm}\label{thm:single}
Let $K$ be a number field. Let $C/K$ be a curve of genus $g$,
and write $J$ for the Jacobian of $C$.
Let $n \ge 2$ and suppose 
\begin{equation}\label{eqn:glb}
	\begin{cases} g>(n-1)^2 & \text{if $n \ge 3$}\\
		g\ge 3 & \text{if $n=2$}.
	\end{cases}
\end{equation}
Suppose that either $J$ is simple,
or $J(K)$ is finite.
If $C$ has infinitely many primitive points of degree $n$,
then there is a degree $n$ morphism $\varphi : C \rightarrow \PP^1$
defined over $K$
such that all but finitely many primitive degree $n$ divisors 
are fibres $\varphi^{*}(\alpha)$ with $\alpha \in \PP^1(K)$.
\end{thm}
We call Theorem~\ref{thm:single} the \lq\lq Single Source Theorem\rq\rq,
since, with finitely many exceptions, all primitive degree $n$ points come
from a single source which the morphism $\varphi : C \rightarrow \PP^1$. 

\begin{thm}\label{thm:primitive}
Let $K$ be a number field. Let $C/K$ be a curve of genus $g$,
and write $J$ for the Jacobian of $C$.
Let $n \ge 2$ and suppose \eqref{eqn:glb} holds.
Suppose that either $J$ is simple,
or that $J(K)$ is finite.
Suppose $C$ has infinitely many degree $n$ points with Galois
group $S_n$ or $A_n$.
Then $C$ has only finitely many degree $n$
points with any other primitive Galois group. 
\end{thm}

We point out that, in both Theorems~\ref{thm:single} and~\ref{thm:primitive},
we may replace the condition \lq $J$ is simple or $J(K)$ is finite\rq\
with the weaker condition \lq $W_n(C)$ does not contain the translate
of a positive rank abelian variety\rq. Here $W_n(C)$ is the image of $C^{(n)}$,
the $n$-symmetric power of $C$, in $J$ under an Abel--Jacobi map.

\medskip

We mention two intermediate results that may be of independent interest.
\begin{thm}\label{thm:coolthm}
Let $K$ be a perfect field. 
Let $n \ge 2$. Let $C$ be a curve of genus $g$ defined over $K$.
Suppose
\begin{equation}\label{eqn:gt}
g \; > \; \frac{(n-1)(n-2)}{2}.
\end{equation}
Let $D$ be a primitive degree $n$ divisor on $C$. Then
$\ell(D) \le 2$.
\end{thm}
Here $L(D)$ donotes the Riemann--Roch space associated
ot $D$, and $\ell(D)$ denotes its dimension.
We believe that Theorem~\ref{thm:coolthm} is the 
first ever example of a relationship
between the Galois group of a divisor,
and its Riemann--Roch dimension.
\begin{thm}\label{thm:singleRR}
Let $K$ be a perfect field.
Let $C/K$ be a curve of genus $g$. Let $n \ge 2$.
Let $D_1$, $D_2$ be two primitive degree $n$ divisors on $C$
with $\ell(D_1)=\ell(D_2)=2$. Suppose 
\begin{equation}\label{eqn:gt2}
g>(n-1)^2.
\end{equation}
Then $D_1$, $D_2$ are linearly equivalent.
\end{thm}

\medskip

The paper is structured as follows. 
In Section~\ref{sec:primact} we review
some standard results on primitive
group actions that are needed later in the paper.
In Section~\ref{sec:coolthm}
we prove Theorem~\ref{thm:coolthm}: if $\ell(D) \ge 3$ and $D$ is primitive
then we show that $C$ is birational to a plane degree $n$
curve which contradicts \eqref{eqn:gt}. In Section~\ref{sec:singleRR}
we prove Theorem~\ref{thm:singleRR}: if $D_1$, $D_2$ 
are inequivalent and primitive then we show that $C$
is birational to an $(n,n)$-curve on $\PP^1 \times \PP^1$
contradicting \eqref{eqn:gt2}. In Section~\ref{sec:singleproof}
we show that Theorem~\ref{thm:single} follows from Theorems~\ref{thm:coolthm},
~\ref{thm:singleRR} and a famous theorem of Faltings on rational points
lying on subvarieties of abelian varieties. In Section~\ref{sec:minimalindex}
we recall the concept of the minimal index of a group action,
and we use recent work of Burness and Guralnick
to study the minimal index of the natural actions of $S_n$ and
$A_n$ on the coset spaces $S_n/H$ and $A_n/H$.
The next two sections (Section~\ref{sec:monodromy}
and Section~\ref{sec:proof}) are devoted to deducing
Theorem~\ref{thm:primitive} from Theorem~\ref{thm:single}.
Indeed Theorem~\ref{thm:single} allows us
to focus on the fibres of a single degree $n$
morphism $\varphi : C \rightarrow \PP^1$ defined over $K$.
We show that this morphism has \lq generic Galois group\rq\ $A_n$ or $S_n$,
and to prove Theorem~\ref{thm:primitive}
it will be enough to show that only finitely many fibres have
primitive Galois groups $\ne A_n$, $S_n$. The map $\varphi$
is not in general a Galois cover, and we will need
 to consider the \lq geometrically connected Galois closure\rq\ 
$\tilde{C} \rightarrow \PP^1$,
which is defined over an extension $L$ of $K$. We show,
using the simplicity of $A_n$, that either $L=K$
or $L/K$ is quadratic. The fibres $\varphi^*(\alpha)$,
for $\alpha \in \PP^1(L)$,
which have any given Galois group $H$ give rise to $L$-points
on some subcover $D_H/L$ of $\tilde{C}$. The proof 
of Theorem~\ref{thm:primitive} boils down to showing
that all these $D_H$ have genus $\ge 2$ and hence
finitely many $L$-points by Faltings' theorem.
We give a formula (Lemma~\ref{lem:monodromy})
for the genus of $D_H$ in terms 
of element indices for a certain coset action. The
proof of this formula makes use
of recent work of Lombardo, Lorenzo Garc\'{\i}a, Ritzenthaler
and Sijsling who determine the ramification data
for the cover $D_H \rightarrow \PP^1$ in terms of
the monodromy data of the original cover $C \rightarrow \PP^1$.
The group theory results of Section~\ref{sec:minimalindex} on minimal
index allow us to deduce that all our $D_H$ have
genus $\ge 2$, and so to complete the proof of Theorem~\ref{thm:primitive}.

\medskip

We are grateful to Gareth Tracey for useful discussions,
and for drawing our attention to the work
of Burness and Guralnick \cite{BurnessGuralnick}
on fixed point ratios.

\section{Primitive Group Actions}\label{sec:primact}
In this section we review some properties of primitive
group actions. This is standard material
(e.g. \cite{DixonMortimer}), and is included
for the convenience of the reader.
Let $G$ be a group acting transitively on a finite set $\Omega$.
Let $\cP$ be a partition of $\Omega$. We say that $\cP$
is $G$-stable if $\sigma(Y) \in \cP$ for all $Y \in \cP$
and all $\sigma \in G$.
We say that the action of $G$ on $\Omega$ is \textbf{primitive}
if the only $G$-stable partitions of $\Omega$
are the trivial ones: $\{\Omega\}$ and $\{ \{\omega\} : \omega \in \Omega \}$.
Here is an equivalent formulation: the action of $G$ on $\Omega$
is imprimitive if and only if there is some $Y \subset \Omega$
such that $2 \le \# Y < \# \Omega$, and for all $\sigma \in G$
either $\sigma (Y)=Y$ or $\sigma(Y) \cap Y = \emptyset$.

\begin{lem}\label{lem:2trans}
Suppose the action of  $G$ on $\Omega$ is $2$-transitive.
Then the action is primitive.
\end{lem}
\begin{proof}
Let $Y$ be a subset of $\Omega$ with at least two elements
and suppose that for all $\sigma \in G$,
either $\sigma(Y)=Y$ or $\sigma(Y) \cap Y=\emptyset$.
We want to show that $Y=\Omega$. 
Let $c \in \Omega$ and we want to show that $c \in Y$.
Let $a$, $b \in Y$
be distinct. We may suppose $c \ne a$, $b$.
As $G$ is $2$-transitive on $\Omega$,
 there is some $\sigma \in G$ such that $\sigma(a)=a$
and $\sigma(b)=c$. As $a \in Y \cap \sigma(Y)$
we have $Y=\sigma(Y)$ and so $c \in Y$.
\end{proof}
It follows from Lemma~\ref{lem:2trans}
that $S_n$ is primitive for all $n$,
and $A_n$ is primitive for $n \ge 3$.

\begin{lem}\label{lem:primmax}
Suppose $\lvert \Omega \rvert \ge 2$.
The action of $G$ on $\Omega$ is primitive if and only
if $\Stab(\omega)$ is maximal for any (and hence all)
$\omega \in \Omega$.
\end{lem}
\begin{proof}
As the action is transitive, any two point stabilizers are
conjugate, and thus if one is maximal then so are all of them.
Let $\omega \in \Omega$. As $\lvert \Omega \rvert \ge 2$,
the stabilizer $\Stab(\omega)$ is a proper subgroup of $G$.
Suppose it is non-maximal, and let $\Stab(\omega) \subsetneq H \subsetneq G$
be a subgroup.
Let $Y=H\omega$. Then 
\begin{equation}\label{eqn:max}
2 \le \underbrace{[H:\Stab(\omega)]}_{\#Y} < [G:\Stab(\omega)] =\#\Omega.
\end{equation}
Suppose $\sigma \in G$ and  $Y \cap \sigma (Y) \ne \emptyset$.
Then, there are $h_1$, $h_2 \in H$ such that $h_1 \omega=\sigma h_2 \omega$,
and so $h_1^{-1} \sigma h_2 \in \Stab(\omega) \subset H$,
so $\sigma \in H$, and hence $\sigma(Y)=(\sigma H) \omega=H \omega=Y$.
Therefore the action is imprimitive.

Conversely, suppose the action is imprimitive, so there is some $Y \subset
\Omega$ satisfying $2 \le \# Y <\#\Omega$, and for all $\sigma \in G$,
either $\sigma(Y) \cap Y=\emptyset$ or $\sigma(Y)=Y$.
Let $\omega \in Y$ and let $H=\{ \tau \in G : \tau(Y)=Y\}$.
If $\sigma \in \Stab(\omega)$ then $\omega \in Y \cap \sigma(Y)$
so $\sigma(Y)=Y$ and so $\sigma \in H$. Hence $\Stab(\omega) \subseteq H$.
Moreover, as $H$ acts transitively on the elements of $Y$,
we have $[H:\Stab(\omega)]=\#Y$, so \eqref{eqn:max} holds,
and therefore $\Stab(\omega)$ is non-maximal.
\end{proof}

\begin{lem}\label{lem:primnorm}
Suppose $G$ acts primitively on $\Omega$. Let $N$ be a normal
subgroup of $G$. Then $N$ acts either transitively 
or trivially on $\Omega$.
\end{lem}
\begin{proof}
We may suppose $\# \Omega \ge 2$.
Let $\omega \in \Omega$. By Lemma~\ref{lem:primmax},
the stabilizer $\Stab(\omega)$ is maximal.
Let 
\[
H=N \Stab(\omega)=\{n k : n \in N,~k\in \Stab(\omega) \}.
\] 
As $N$ is normal, $H$ is a subgroup of $G$,
and since $\Stab(\omega)$ is maximal,
$H=\Stab(\omega)$ or $H=G$. 
Suppose first that $H=G$. Then 
$\Omega=G\omega=H\omega=N \omega$,
so $N$ acts transitively.
Suppose instead that $H=\Stab(\omega)$. Then
$N \subseteq \Stab(\omega)$. As $N$ is normal and all point
stabilizers are conjugate, we see that $N$ is contained
in all point stabilizers and so acts trivially.
\end{proof}

\section{Primitivity and Riemann--Roch Dimension}\label{sec:coolthm}
In this section we prove Theorem~\ref{thm:coolthm}.
\begin{lem}\label{lem:critical}
Let $K$ be a perfect field and let $C/K$
be a curve.  Let $D$ be a primitive divisor. 
Let $f \in L(D)$ be non-constant.
Suppose there is a (possibly singular)
curve $C^\prime/K$, and rational maps
$\varphi: C \dashrightarrow C^\prime$ and $\psi: C^\prime\dashrightarrow \PP^1$
defined over $K$ such that $f=\psi \circ \varphi$. Then
$\deg(\varphi)=1$ or $\deg(\psi)=1$.
\end{lem}
\begin{proof}
Write $\divv_\infty(f)$ for the divisor of poles of $f$.
As $f$ is non-constant and belongs to $L(D)$
we have $0<\divv_\infty(f) \le D$. However, $D$ is irreducible,
therefore $\divv_\infty(f)=D$.

Now let $\pi : C^{\prime\prime} \rightarrow C^\prime$ be
the normalization of $C^\prime$. The map $\pi$ is birational,
and we write $u=\pi^{-1} \circ \varphi$, and $v=\psi \circ \pi$.
As $C$ and $C^{\prime\prime}$ are proper, $u : C \rightarrow C^{\prime\prime}$
and $v: C^{\prime \prime} \rightarrow \PP^1$ are morphisms
defined over $K$.
Consider the following commutative diagram.
\[
\begin{tikzcd}
& C^{\prime\prime} \arrow[d, "\pi"] \arrow[rd,"v"]& \\
C \arrow[ru,"u"] \arrow[r, "\varphi" below, dashed] & C^\prime \arrow[r, "\psi" below, dashed] & \PP^1
\end{tikzcd}
\]
We note that $f=\psi \circ \varphi=v \circ u$.
In particular, 
	$D=f^{*}(\infty)=u^{*} (v^*(\infty))$.

Write $r=\deg(u)$ and $s=\deg(v)$.
Write $v^*(\infty)=Q_1+\cdots+Q_s$. 
Note that
\[
	\{ u^{-1}(Q_i) : i \in \{1,\dotsc,s\}\}
\]
is a partition of the points in $D$, into $s$ subsets of size $r$, that is Galois-stable.
As $D$ is primitive, either $r=1$ or $s=1$. However $r=\deg(\varphi)$ and $s=\deg(\psi)$,
completing the proof.
\end{proof}

\begin{proof}[Proof of Theorem~\ref{thm:coolthm}]
Suppose $\ell(D) \ge 3$. Then there are $f$, $g \in K(C)$
such that $1$, $f$, $g$ are linearly independent
elements of $L(D)$. 
Let $V$ be the subspace of $L(D)$ spanned by $1$, $f$, $g$,
and consider the corresponding linear system
\begin{equation}\label{eqn:lin}
	\{D+\divv(h) \; : \;  h \in V\}.
\end{equation}
We claim that  \eqref{eqn:lin} is base point free.
Indeed, let $D_0$ be the base locus of \eqref{eqn:lin}.
Thus $D_0$ is a $K$-rational divisor
and $D_0 \le D$. Since $D$ is irreducible, either $D_0=0$
or $D_0=D$. If $D_0=D$ then all elements of the linear system \eqref{eqn:lin}
are equal to $D$, which makes all $h \in V$ constant, giving a contradiction.
Thus $D_0=0$ establishing our claim.
We let
\[
\varphi : C \xrightarrow{|V|} \PP^2, \qquad \varphi=(f : g : 1),
\]
and let $C^\prime$ be the 
image of $C$ in $\PP^2$ under $\varphi$, which is a geometrically irreducible
curve defined over $K$, but may be singular.
We also denote by $\varphi$ the morphism $C \rightarrow C^\prime$.
Write 
\[
	\psi : C^\prime \rightarrow \PP^1, \qquad \psi(x:y:z)=(x:z).
\]
Then $\psi \circ \varphi=f$. 
By the argument at the beginning of the proof of Lemma~\ref{lem:critical},
we have $\divv_\infty(f)=D$, so $\deg(f)=\deg(D)=n$.
Moreover, for any $(a:b) \in \PP^1$, we observe that
$\psi^*(a:b)$ is the hyperplane
section $C^\prime \cap \{bx-az=0\}$, and so $n/\deg(\varphi)=\deg(\psi)$
equals the degree of $C^\prime$ as plane curve.
Applying Lemma~\ref{lem:critical} to $\psi \circ \varphi=f$ gives
$\deg(\varphi)=1$ or $\deg(\psi)=1$.

However, if $\deg(\psi)=1$ then
$C^\prime$ is a line which
contradicts the linear independence of $1$, $f$, $g$.
Thus $\deg(\varphi)=1$, and so the map $\varphi : C \rightarrow C^\prime$
is birational. 
Hence the geometric genus
of $C^\prime$ is $g$.  Since $C^\prime$ has degree $n$,
its arithmetic genus is
$(n-1)(n-2)/2$. As the geometric genus is
bounded by the arithmetic genus, we have
that $g \le (n-1)(n-2)/2$. This contradicts \eqref{eqn:gt}.
\end{proof}

\section{Proof of Theorem~\ref{thm:singleRR}}\label{sec:singleRR}

As $\ell(D_i)=2$, we may choose non-constant $f_i \in L(D_i)$.
Since $D_i$ are irreducible,
	$\divv_\infty(f_i)=D_i$ as before. In particular, $\deg(f_i)=n$.
Let 
\[
	\varphi \; : \; C \rightarrow \PP^1 \times \PP^1, \qquad \varphi=(f_1,f_2),
\]
and let $C^\prime=\varphi(C)$. Then $C^\prime/K$ is an irreducible
but possibly singular curve lying on $\PP^1 \times \PP^1$;
we also denote the map $C \rightarrow C^\prime$ by $\varphi$.
Let $\pi_1$, $\pi_2 : \PP^1 \times \PP^1 \rightarrow \PP^1$
denote projection onto the first and second factor respectively.
Let $\mu_i= \pi_i \vert_{C^\prime} : C^\prime \rightarrow \PP^1$.
Then $f_i=\mu_i \circ \varphi$. By Lemma~\ref{lem:critical},
there are two possibilities:
\begin{enumerate}[(I)]
		\item either $\deg(\varphi)=1$ and $\deg(\mu_1)=\deg(\mu_2)=n$;
		\item or $\deg(\varphi)=n$ and $\deg(\mu_1)=\deg(\mu_2)=1$.
\end{enumerate}
Suppose that (I) holds. Then $\varphi$ is a birational map,
and so $C$ and $C^\prime$ have the same geometric genus $g$.
Moreover, $C^\prime$ is a curve of bidegree $(n,n)$ on $\PP^1 \times \PP^1$
and therefore has arithmetic genus $(n-1)^2$ (see \cite[Exercise III.5.6]{Hartshorne}).
Thus $g \le (n-1)^2$ giving a contradiction.

Therefore (II) holds. Thus $\mu_1$, $\mu_2$ are birational, and we have a commutative
diagram of morphisms
\[
\begin{tikzcd}
	& C \arrow[ld, "f_1" above] \arrow[d, "\varphi"] \arrow[rd, "f_2"] & \\
	\PP^1 & C^\prime \arrow[l, "\mu_1"] \arrow[r, "\mu_2" below] & \PP^1
\end{tikzcd}
\]
Write $\mu=\mu_2 \circ \mu_1^{-1}$. Then $\mu : \PP^1 \rightarrow \PP^1$
is an automorphism satisfying $f_2=\mu \circ f_1$.
Thus $f_2^*=f_1^* \circ \mu^*=f_1^* \circ \mu^{-1}$.
Let $\alpha=\mu^{-1}(0)$. Then
\[
		\divv(f_2)=f_2^*(0)-f_2^*(\infty)=f_2^*(0)-D_2,
		\qquad
		\divv(f_1-\alpha)=f_1^*(\alpha)-f_1^*(\infty)=f_2^*(0)-D_1.
\]
Hence $D_2-D_1=\divv((f_1-\alpha)/f_2)$ establishing the theorem.

\section{Proof of Theorem~\ref{thm:single}}\label{sec:singleproof}
We shall need the following famous theorem due to Faltings \cite[Theorem 1]{FaltingsDio}.
\begin{thm}[Faltings]\label{thm:Faltings}
Let $A$ be an abelian variety defined over a number field $K$.
Let $V \subset A$ be a subvariety containing no translate
of a positive dimensional abelian subvariety of $A$. Then $V(K)$
is finite.
\end{thm}

\begin{proof}[Proof of Theorem~\ref{thm:single}]
Write $C^{(n)}$ for the $d$-th symmetric power of $C$.
We recall that $C^{(n)}(K)$ parametrizes
effective degree $n$ divisors on $C$, and 
write $C^{(n)}_{\prim}(K)$ for the subset of primitive 
degree $n$ divisors. By assumption $C^{(n)}_{\prim}(K)$
is infinite. Fix $D_0 \in C^{(n)}(K)$, and let
\[
		\iota \; : \; C^{(n)} \rightarrow J, \qquad D \mapsto [D-D_0]
\]
be the corresponding Abel--Jacobi map. Write $W_n=\iota(C^{(n)})$.
We note that $\dim(W_n)=n<g=\dim(J)$ by assumption~\eqref{eqn:glb}.

We claim that $W_n(K)$ is finite. This is trivially true
if $J(K)$ is finite. Suppose instead that $J$ is simple.
Then the finiteness of $W_n(K)$ follows immediately from
Faltings' Theorem. This establishes our claim.

Now write $W_n(K)=\{Q_1,\dotsc,Q_r\}$. 
Let $D_1^\prime,\dotsc,D_r^\prime \in C^{(n)}(K)$
be effective divisors satisfying $\iota(D_i^\prime)=Q_i$.
Then
\[
	C^{(n)}(K) \; = \; 
	\bigcup_{i=1}^r \,  \lvert D_i^\prime \rvert;
\]
here, $\lvert D \rvert$ denotes the complete linear system
corresponding to $D$:
\[
	\lvert D \rvert \; = \; \{ D+\divv(f) \; : \;  f \in L(D)\}.
\]
Hence,
\[
	C^{(n)}_{\prim}(K) \; \subseteq \; 
	\bigcup_{j=1}^m \,  \lvert D_j \rvert;
\]
where $\lvert D_1 \rvert,\dotsc,\lvert D_m \rvert$ are the complete linear system 
among $\lvert D_1^\prime \rvert,\dotsc,\lvert D_r^\prime \rvert$
containing a primitive divisor. Recall that $D^\prime \in \lvert D\rvert$
if and only if $\lvert D^\prime \rvert=\lvert D \rvert$. Hence, we may
suppose that $D_1,\dotsc,D_m$ are primitive.
We now apply Theorem~\ref{thm:coolthm}. This tells us that $\ell(D_i)=1$ or $2$
for $i=1,\dotsc,m$. Moreover, Theorem~\ref{thm:singleRR} tells us
that $\ell(D)=2$ for at most one divisor $D$ among
$D_1,\dotsc,D_m$.
If $\ell(D)=1$ then $\lvert D \rvert=\{D\}$. Since $C^{(n)}_{\prim}(K)$
is infinite, we deduce, after permuting the $D_i$, that
\[
	C^{(n)}_{\prim}(K) \; \subseteq \; 
	\{D_1,\dotsc,D_{m-1}\} \, \cup \, \lvert D_m \rvert,
\]
where $\ell(D_m)=2$.  Let $\varphi \in L(D_m)$ be a non-constant function,
which we regard as a morphism $\varphi : C \rightarrow \PP^1$
satisfying $\varphi^*(\infty)=D_m$.
If $D \in \lvert D_m \rvert$, and $D \ne D_m$, 
then $D=D_m+\divv(\varphi-\alpha)$ for some $\alpha \in K$,
and so $D=\varphi^*(\alpha)$. This completes the proof.
\end{proof}

\section{Bounds for the minimal index of a group action}\label{sec:minimalindex}
Given a group $G$ acting on a finite
set $\Omega$, and an element $g \in G$,
we define the \textbf{fixed point ratio} of $g$ to be
\begin{equation}\label{eqn:fpr}
	\fpr(g,\Omega) :=  \frac{\# \Fix(g,\Omega)}{\# \Omega}
\end{equation}
where $\Fix(g,\Omega)$ denotes the set of points in $\Omega$ fixed by $g$.
We define the \textbf{index of $g$} to be
\begin{equation}\label{eqn:index}
	\ind(g,\Omega) := \# \Omega \, - \, \# \Orb(g,\Omega)
\end{equation}
where $\Orb(g,\Omega)$ is the set of orbits of $g$
acting on $\Omega$. We define the \textbf{minimal index of $G$
acting on $\Omega$} by 
\[
\ind(G,\Omega) := \min \{\ind(g,\Omega) \; : \; g \in G, \; g \ne 1 \}.
\]
The minimal index was originally introduced by
Malle \cite{Malle} in the context of counting
number fields with a given Galois group.
We note the following elementary relation between the index
and the fixed point ratio.
\begin{lem}\label{lem:indfpr}
$\displaystyle	\ind(g,\Omega) \; \ge \; \frac{\# \Omega}{2} (1-\fpr(g))$.
\end{lem}
\begin{proof}
We partition $\Omega$ into $\Fix(g,\Omega)$
	and $\Omega \setminus \Fix(g,\Omega)$.
Every element of $\Fix(g,\Omega)$ is an orbit of $g$,
but every element of $\Omega \setminus \Fix(g,\Omega)$
belongs to an orbit of size at least $2$. Thus,
\[
	\begin{split}
		\# \Orb(g,\Omega)  & \le \# \Fix(g,\Omega) + 
		\frac{1}{2} \left(\#\Omega-\# \Fix(g,\Omega) \right)\\
		&= \frac{1}{2} \left( \# \Omega+\# \Fix(g,\Omega) \right)\\
		&= \frac{\# \Omega}{2} \left( 1+\# \fpr(g,\Omega) \right),
	\end{split}
\]
using \eqref{eqn:fpr}. The lemma follows from \eqref{eqn:index}.
\end{proof}
Given a group $G$ and a subgroup $H$, we write $G/H$ for the set
of left cosets of $H$ in $G$, and we give $G/H$ the natural
$G$-action
\[
	G \times G/H \rightarrow G/H, \qquad (g,g^\prime H) \mapsto g g^\prime H.
\]
It is clear that this action is transitive.
As we shall see, the minimal index of $G$ acting on $G/H$ arises 
naturally in the context
of estimating the genera of subcoverings
of a $G$-covering of the projective line. We shall need
an estimate for $\ind(A_n,A_n/H)$ and $\ind(S_n,S_n/H)$,
for certain subgroups $H$, and for this we
need the following recent theorem of Burness and Guralnick
\cite{BurnessGuralnick}.  
\begin{thm}[Burness and Guralnick]\label{thm:BG}
Let $G=A_n$ with $n \ge 5$,
or $G=S_n$ with $n=5$ or $n \ge 7$.
Let $\Omega$ be a finite set on which $G$ acts primitively
and faithfully.
Let $g \in G$ be an element of prime order $r$. Then
$\fpr(g,\Omega) \le 1/r$ unless, for some $1 \le \ell < n/2$,
the $G$-set $\Omega$ is isomorphic to the $G$-set
\[
	\Omega_\ell \; =\; \{A \subseteq \{1,\dotsc,n\} \; :\;  \lvert A \rvert=\ell\},
\]
endowed with the natural $G$-action.
\end{thm}
\noindent \textbf{Remark.}
Theorem~\ref{thm:BG} is in fact
a special case of a recent result of Burness and Guralnick
\cite[Corollary 3]{BurnessGuralnick} which applies more generally
to primitive faithful actions of almost simple groups $G$. 
A group $G$ is said to be \textbf{almost simple}
if it contains a simple non-abelian subgroup $H$ such that
$H \subseteq G \subseteq \Aut(H)$. For $n \ge 5$
the group $A_n$ is simple and hence almost simple.
Also, for $n=5$ or $n \ge 7$, we have that $\Aut(A_n)=S_n$,
therefore $S_n$ is almost simple.

The more general
result of Burness and Guralnick 
includes a substantial list of exceptions where $G$ is
a classical group. We avoid having to consider these by specialising
our group to $G=A_n$ or $S_n$ as these are the only cases we need.

\begin{lem}\label{lem:faithful}
Let $n \ge 5$.
\begin{enumerate}[(i)]
\item Let $H \ne A_n$, $S_n$ be a subgroup of $S_n$.
Then $S_n$ acts faithfully on $S_n/H$.
\item Let $H$ be a proper subgroup of $A_n$.
Then $A_n$ acts faithfully on $A_n/H$.
\end{enumerate}
\end{lem}
\begin{proof}
Let $H \ne A_n$, $S_n$ be a subgroup of $S_n$.
Let $g \in S_n$ and suppose $g \sigma H=\sigma H$ for all $\sigma \in S_n$.
Then $\sigma^{-1} g \sigma \in H$ for all $\sigma \in S_n$.
Write
\[
H^\prime \; = \; \langle \sigma^{-1} g \sigma \; : \; \sigma \in S_n \rangle.
\]
Then $H^\prime$ is a normal subgroup of $S_n$
and so is equal to $1$, $A_n$ or $S_n$.
But $H^\prime \subseteq H$, so $H^\prime=1$
and so $g=1$,
establishing (i). The proof of (ii) is similar.
\end{proof}

\begin{lem}\label{lem:fpr}
Let $n \ge 5$.
	\begin{enumerate}[(I)]
		\item 
		Let $H$ be a maximal transitive subgroup of $A_n$.
		Let $g \in A_n$, $g \ne 1$. Then $\fpr(g,A_n/H) \le 1/2$.
		\item 
		Let $H \ne A_n$ be a maximal transitive subgroup of $S_n$.
		Let $g \in S_n$, $g \ne 1$. Then $\fpr(g,S_n/H) \le 2/3$.
		\item 
		Let $H$ be a maximal transitive subgroup of $A_n$.
		Let $g \in S_n$, $g \ne 1$. Then $\fpr(g,S_n/H) \le 3/4$.
	\end{enumerate}
\end{lem}
\begin{proof}
We consider (I) first, and thus the action of $A_n$ on $A_n/H$. 
Note that $H$ is the stabilizer of the trivial coset $H \in A_n/H$. 
By assumption $H$ is maximal.
Therefore, by Lemma~\ref{lem:primmax},
 the action of $A_n$ on $A_n/H$ is primitive.
Moreover, it is faithful by Lemma~\ref{lem:faithful}. 
We apply Theorem~\ref{thm:BG}. We claim that $A_n/H$ is not
isomorphic to $\Omega_\ell$ as an $G$-set (with $G=A_n$). Observe that
the stabilizer of $\{1,\dotsc,\ell\} \in \Omega_\ell$
is $A_n \cap (S_\ell \times S_{n-\ell})$. If $A_n/H$ is isomorphic
to $\Omega_\ell$ as an $G$-set, then 
$H$ must be conjugate inside $A_n$ to $A_n \cap (S_\ell \times S_{n-\ell})$.
This contradicts the transitivity of $H$, and establishes our claim.

Now let $g \in A_n$, $g \ne 1$. Then, for some positive integer $m$,
$g^m$ has prime order $r$ (say), and so $\fpr(g^m,A_n/H) \le 1/r \le 1/2$
by Theorem~\ref{thm:BG}. But clearly $\Fix(g,A_n/H) \subseteq \Fix(g^m,A_n/H)$.
Thus $\fpr(g,A_n/H) \le \fpr(g^m,A_n/H) \le 1/2$ completing the proof of (I).

For $n \ne 6$, the proof of (II) is almost identical to the proof of (I),
and in fact gives $\fpr(g,S_n/H) \le 1/2$.
For $n=6$ we may not apply Theorem~\ref{thm:BG}. However we
settled this case by brute force enumeration
using the computer algebra package \texttt{Magma} \cite{magma}.
The group $A_6$ has three maximal transitive subgroups $H \ne A_n$,
with orders $48$, $72$ and $120$. These
yield bounds for $\fpr(g,S_n/H)$
which are respectively $7/15$, $2/5$ and $2/3$. This completes
the proof of (II).

Next we turn our attention to (III). Note that in
this case $H$ is maximal in $A_n$ and not in $S_n$.
Therefore, the action of $S_n$ on $S_n/H$ is not primitive
and so we may not apply Theorem~\ref{thm:BG}.
Let $g \in S_n$, $g \ne 1$.
and suppose that $\fpr(g,S_n/H) \ne 0$.
Then $g (\sigma H)=\sigma H$ for some $\sigma \in S_n$.
Thus $\sigma^{-1} g \sigma \in H \subset A_n$. As $A_n$ is normal in $S_n$ we have 
that $g \in A_n$. Now by (I),
\[
		\# \Fix(g,A_n/H) \; \le \; \frac{1}{2} \cdot \# [A_n:H].
\]
Clearly,
\[
	\Fix(g,S_n/H) \; \subseteq \; \Fix(g,A_n/H) \cup ((S_n/H) \setminus (A_n/H)).
\]
Hence
\[
	\# \Fix(g,S_n/H) \; \le \; \frac{1}{2} \cdot \# [A_n:H] + [S_n:H]-[A_n:H] \; = \;
	\frac{3}{4} [S_n:H],
\]
proving (III).
\end{proof}

\begin{lem}\label{lem:ind}
Let $n \ge 5$.
	\begin{enumerate}[(I)]
		\item 
		Let $H$ be a maximal transitive subgroup of $A_n$.
			Then 
			\[
				\ind(A_n,A_n/H)\ge [A_n:H]/4.
			\]
		\item 
		Let $H \ne A_n$ be a maximal transitive subgroup of $S_n$.
			Then 
			\[
				\ind(S_n,S_n/H) \ge [S_n:H]/6.
			\]
		\item 
		Let $H$ be a maximal transitive subgroup of $A_n$.
			Then 
			\[
				\ind(S_n,S_n/H) \ge [S_n:H]/8.
			\]
	\end{enumerate}
\end{lem}
\begin{proof}
	The lemma follows from Lemma~\ref{lem:fpr}
	and Lemma~\ref{lem:indfpr}.
\end{proof}

\section{Galois Theory and Specializations}\label{sec:monodromy}
Let $K$ be a number field. 
Let $\varphi : C \rightarrow \PP^1$ be a morphism of
curves defined over $K$.
Note that $K(C) \cap \overline{K}=K$,
as $C$ is geometrically connected.
We write $\Ram(\varphi) \subset C$ for the set
of ramification points of $C$. The 
\textbf{set of branch values} for $\varphi$ is 
defined by
\[
	\BV(\varphi) \; = \; \{ \varphi(P) \; : \; P \in \Ram(\varphi)\}
	\; \subset \; \PP^1.
\]

Let $\KK$ be the Galois closure of the
function field extension $K(C)/K(\PP^1)$ induced by $\varphi$.
Write $n=\deg(\varphi)$. Then we may naturally identify $G^\prime:=\Gal(\KK/K(\PP^1))$
with a transitive subgroup of $S_n$. In what follows, when 
we speak of subgroups of $G^\prime$ being transitive or primitive,
it is with respect to the action on $\{1,2,\dotsc,n\}$.
\begin{lem}\label{lem:residue}
Let $\alpha \in \PP^1(K)$, and consider $\alpha$ as a place of $\PP^1$.
Let $\cP$
be a place of $\KK$ above $\alpha$. Then $K(\cP)/K$
is a Galois extension with Galois group isomorphic to the
decomposition group
\[
G^\prime_\cP = \{ \sigma \in G^\prime \; : \; \sigma(\cP)=\cP\}.
\]
\end{lem}
\begin{proof}
For this, see \cite[Theorem III.8.2]{Stichtenoth}.
However, we will sketch some of the ideas in the proof of
Lemma~\ref{lem:obvious}.
\end{proof}

Let $L=\KK \cap \overline{K}$, which is 
a finite Galois extension of $K$. 
Let $G=\Gal(\KK/L(\PP^1))$.
Then we obtain an exact sequence of Galois groups
\begin{equation}\label{eqn:exact}
	1 \rightarrow \underbrace{\Gal(\KK/L(\PP^1))}_{G}
	\rightarrow \underbrace{\Gal(\KK/K(\PP^1))}_{G^\prime} \rightarrow
	\underbrace{\Gal(L(\PP^1)/K(\PP^1))}_{\cong \Gal(L/K)} \rightarrow 1.
\end{equation}
We note that
$\KK/L(\PP^1)$ is  regular in the sense
that $\KK \cap \overline{L}=L$. Therefore $\KK=L(\tilde{C})$
where $\tilde{C}$ is a (geometrically connected) curve defined over $L$.
The inclusions $L(\PP^1) \subseteq L(C) \subseteq L(\tilde{C})$
correspond to morphisms
\[
	\tilde{C} \rightarrow C \xrightarrow{\varphi}{\PP^1}.
\]
and we write $\mu : \tilde{C} \rightarrow \PP^1$ for the composition
which is defined over $L$. We may naturally identify
$G$ with automorphisms of the cover $\mu$.

Now let $H$ be a subgroup of $G$. Write $\KK^H$ for the subfield
of $\KK$ fixed by $H$.
The function field extension $\KK^H/L(\PP^1)$
corresponds to a morphism of curves
$\pi_H : D_H \rightarrow \PP^1$ defined over $L$, where $L(D_H)=\KK^H$.
Note that we have the following commutative diagram of 
morphisms
\begin{equation}\label{eqn:tikz}
\begin{tikzcd}
	& \tilde{C} \arrow[ld] \arrow[rd,"\eta_H"] \arrow[dd,"\mu"] & \\
	C \arrow[rd, "\varphi"] &  & D_H \arrow[ld, "\pi_H"] \\
	& \PP^1 &
\end{tikzcd}
\end{equation}
We note that $\BV(\mu)=\BV(\varphi)$ (see for example \cite[Corollary III.8.4]{Stichtenoth}).
It follows that $\BV(\pi_H) \subseteq \BV(\varphi)$.

We shall use \eqref{eqn:tikz} to study fibres of the map $\varphi$
with certain Galois group. 
The curve $D_H$ is important to us because of the following
standard result.
\begin{lem}\label{lem:obvious}
Let $\cP \in \tilde{C}$ be an algebraic point with $\mu(\cP) \in \PP^1(L) \setminus \BV(\varphi)$.
Write
\[
	G_{\cP} \; := \; \{\sigma \in G \; : \; \sigma(\cP)=\cP \}
\]
for the decomposition group of $\cP$. Let $H$ be a subgroup of $G$,
and suppose $G_{\cP} \subseteq H$. Then $\eta_H(\cP) \in D_H(L)$.
\end{lem}
\begin{proof}
The lemma is implicit in most proofs of Hilbert's Irreducibility
Theorem, e.g. \cite[Proposition 3.3.1]{SerreTopics}, but 
we give a proof as it helps make ideas precise.
Note that $\cP$ is unramified in $\mu$.
Since $\KK/L(\PP^1)$ is a Galois extension, the
extension $L(\cP)/L$ is Galois, and its Galois group
can be identified with $G_{\cP}$ in a natural way;
see for example \cite[Theorem 3.8.2]{Stichtenoth}.
We shall in fact need some of the details of this
identification, which we now sketch.
Write
\begin{equation}\label{eqn:vr}
	\OO_{\cP}=\{f \in \KK \; : \; \ord_{\cP}(f) \ge 0\}, \qquad
	\fm_{\cP}=\{f \in \KK \; : \; \ord_{\cP}(f) > 0\},
\end{equation}
for the valuation ring of $\cP$ and its maximal ideal. Then $L(\cP)$
may be identified with $\OO_{\cP}/\fm_{\cP}$ via the well-defined
map
\[
	\OO_{\cP}/\fm_{\cP} \rightarrow L(\cP), \qquad f+\fm_{\cP} \mapsto f(\cP).
\]
Let $\sigma \in G_{\cP}$, and let $f \in \KK$.
Then
\[
	\ord_{\cP}(\sigma(f)) =\ord_{\sigma^{-1}(\cP)}(f)=\ord_{\cP}(f).
\]
It follows that $\sigma(\OO_\cP)=\OO_\cP$ and $\sigma(\fm_\cP)=\fm_\cP$.
Hence $\sigma \in G_{\cP}$ induces a well-defined automorphism
of $\OO_{\cP}/\fm_{\cP}=L(\cP)$ given by $\sigma(f+\fm_\cP)=\sigma(f)+\fm_\cP$.
Since $L \subseteq L(\PP^1)$ which is fixed by $G$,
the automorphism on $L(\cP)$ induced by $\sigma$ fixes $L$.
We have now constructed a homomorphism $G_{\cP} \rightarrow \Aut(L(\cP)/L)$.
It turns out \cite[Theorem III.8.2]{Stichtenoth}, since $\KK/L(\PP^1)$
is Galois, that
$L(\cP)/L$ is Galois, and that the homomorphism constructed is in fact an isomorphism
$G_{\cP} \xrightarrow{\sim} \Gal(L(\cP)/L)$.

Now write $R=\eta_H(\cP)$. We would like to show that $\eta_H(R) \in D_H(L)$.
It is enough to show that $g(R) \in L$ for all $g \in \OO_{R}$.
However, $g(R)=f(\cP)$, where $f=\eta_H^*(g) \in \OO_{\cP}$.
Thus we need to show that $f(\cP) \in L$. 
This is equivalent to showing that $\sigma (f(\cP))=f(\cP)$
for all $\sigma \in \Gal(L(\cP)/L)$, which is equivalent 
to showing that $\sigma(f+\fm_{\cP})=f+\fm_{\cP}$
for all $\sigma \in G_{\cP}$. However, by the construction
of the function field of $D_H$, we see that $\sigma(f)=f$
for all $\sigma \in H \supseteq G_{\cP}$.
This completes the proof.
\end{proof}

\begin{lem}\label{lem:IJ}
Let $P \in C/K$ be a primitive degree $n$ point,
with $\varphi(P)=\alpha \in \PP^1(K) \setminus \BV(\varphi)$,
and suppose $P \notin C(L)$.
Let $P_1=P,P_2,\dotsc,P_n$ be the Galois orbit of $P$,
and let 
\[
	\rho : \Gal(\overline{K}/K) \rightarrow \Sym(\{P_1,\dotsc,P_n\})
\]
be the permutation representation obtained from the Galois
action of $\Gal(\overline{K}/K)$ on the orbit. 
Let $I=\rho(\Gal(\overline{K}/K))$  and $J=\rho(\Gal(\overline{K}/L))$;
these are the Galois groups of $P$ over $K$ and $L$ respectively.
Then the following hold.
\begin{enumerate}[(i)]
	\item $J$ is a non-trivial normal subgroup of $I$
	and is a transitive subgroup of $\Sym(\{P_1,\dotsc,P_n\}) \cong S_n$.
	\item Let $\cP \in \tilde{C}$ be above $P$. Then
	$G_{\cP} \subseteq G$ is conjugate to $J$ when
		both are regarded as subgroups of $S_n$.
	\item There is some subgroup $H$ of $G$, conjugate to $J$ in $S_n$,
	such that $R \in D_H(L)$ where $R=\eta_H(\cP)$. 
\end{enumerate}	
\end{lem}
\begin{proof}
By assumption, $\rho(\Gal(\overline{K}/K))$ is a primitive
subgroup of $\Sym(\{P_1,\dotsc,P_n\}) \cong S_n$.
Since $L/K$ is Galois,
$J=\rho(\Gal(\overline{K}/L))$
is a normal subgroup of $\rho(\Gal(\overline{K}/K))$.
By Lemma~\ref{lem:primnorm}, the group $J \subset \Sym(\{P_1,\dotsc,P_n\})$
is either trivial or transitive. 
However, since $P \notin C(L)$,
the group $J$ is non-trivial,
and therefore transitive.
This proves (i).

Recall that $\varphi(P)=\alpha \in \PP^1(K) \setminus \BV(\varphi)$.
Since $P$ has precisely $n$ conjugates, and $\deg(\varphi)=n$,
we see that the fibre $\varphi^*(\alpha)$ consists
of $P_1,\dotsc,P_n$, each with multiplicity $1$.
By composing $\varphi$ with a suitable automorphism of $\PP^1$
we may suppose that $\alpha \in \Aff^1(K)=K$. We shall find it 
convenient to think of $\varphi$ as an element of $K(C)$,
and with this identification we have $K(\PP^1)=K(\varphi) \subseteq K(C)$.
The extension $K(C)/K(\varphi)$ has degree $n$.

Write
\[
	\OO_{P}=\{ h \in K(C) \; : \; \ord_P(h) \ge 0\}, 
        \qquad
	\fm_{P}=\{ h \in K(C) \; : \; \ord_P(h) \ge 1\},
\]
for the valuation ring of $P$ and its maximal ideal. Then the residue field
$\OO_{D}/\fm_{D}$ can be identified with $K(P)$ where
the identification is given by $g +\fm_{D} \mapsto g(P)$.
Now fix $\theta \in K(P)$ such that $K(P)=K(\theta)$.
Note that $[K(\theta) : K]=n$ since $P$ has degree $n$.
Then there is some $g \in \OO_{P}$ such that
$g(P)=\theta$. As $g \in K(C)$ and $K(C)$ has degree $n$ over $K(\varphi)$,
there is a polynomial $F(U,V) \in K[U,V]$,
\begin{equation}\label{eqn:FUV}
        F(U,V)=\sum_{i=1}^m a_i(V) U^i, \qquad a_i(V) \in K[V]
\end{equation}
of degree $m \mid n$, such that $\gcd(a_0(V),\dotsc,a_m(V))=1$,
and $F(g,\varphi)=0$.
Now, $F(\theta,\alpha)=F(g(P),\varphi(P))=0$, and so $\theta$ is a root of the polynomial $F(U,\alpha) \in \Q[U]$;
this polynomial is non-zero as $\gcd(a_0(V),\dotsc,a_m(V))=1$.
As $\theta$ has degree $n$ over $K$, it follows that $m=n$, and 
that $F(U,V)$ is irreducible over $K(V)$.
In particular $F(U,V)=0$ is a (possibly singular) plane model for $C/K$,
and the map $\varphi$ is given by $(u,v) \mapsto v$. As $C$
is absolutely irreducible, $F(U,V)$ is irreducible over $\overline{K}$.
Let $g_1=g,g_2,\dotsc,g_n$ be the roots of $F(U,\varphi)=0$ in $\KK$;
then $\KK=K(\varphi)(g_1,\dotsc,g_n)$. In particular, $G^\prime=\Gal(\KK/K(C))$
may be identified as a transitive subgroup of $\Sym(g_1,\dotsc,g_n) \cong S_n$.

Let $\theta_1=\theta,\theta_2,\dotsc,\theta_n$ be the roots of $F(U,\alpha)=0$,
which are distinct since $K(\theta)/K$ has degree $n$.
We see that the affine plane model $F(U,V)=0$ for $C$
has $n$ distinct points $(\theta_1,\alpha),\dotsc,(\theta_n,\alpha)$
above $\alpha \in \PP^1$. 
However, the smooth model $C$ has precisely $n$ points $P_1,\dotsc,P_n$
above $\alpha \in \PP^1$.
After relabeling
we may identify $P_i=(\theta_i,\alpha)$. 
Next we consider the action of $\Gal(\overline{K}/L)$ on 
$P_1,\dotsc,P_n$, and recall that $\alpha \in K \subseteq L$.
It follows that $J$ is conjugate to $\Gal(L(\theta_1,\dotsc,\theta_n)/L)$ when
we consider $J$ as a subgroup of $\Sym(\{P_1,\dotsc,P_n\}) \cong S_n$
and $\Gal(L(\theta_1,\dotsc,\theta_n)/L)$ as a subgroup of $\Sym(\{\theta_1,\dotsc,\theta_n\}) \cong S_n$.

Let $\cP$ be a point of $\tilde{C}$ above $P$.
Thus 
\[
	g_1(\cP)=g(P)=\theta, \qquad \mu(\cP)=\varphi(P)=\alpha \in K \subseteq L.
\]
As in the proof of Lemma~\ref{lem:obvious},
the extension $L(\cP)/L$ is Galois. Since $\theta=g_1(\cP) \in L(\cP)$
we see that $\theta_1,\dotsc,\theta_n \in L(\cP)$.
In particular, for each $i$, there is an automorphism $\sigma \in L(\cP)/L$
such that $\sigma(\theta)=\theta_i$. Recalling
the natural identification of $\Gal(L(\cP)/L)$ with $G_\cP$,
we see that there is an automorphism $\sigma^\prime \in G_\cP$
such that $\sigma^\prime(g_1)(\cP)=\sigma(\theta)=\theta_i$.
However, $\sigma^\prime(g_1)$ is a root of $F(U,\varphi)$
and is equal to one of the $g_j$. Thus, $g_i(\cP)=\theta_i$,
after suitably reordering 
$g_1,\dotsc,g_n$.  Since $g_1,\dotsc,g_n$ generate $\KK$, we conclude
that $L(\cP)=L(\theta_1,\dotsc,\theta_n)$,
and that $G_\cP$ is conjugate to $\Gal(L(\theta_1,\dotsc,\theta_n)/L)$
which is in turn conjugate to $J$. This proves (ii).

Finally, letting $H=G_\cP$, we deduce (iii) from
Lemma~\ref{lem:obvious}.
\end{proof}

\bigskip

Of course, if $\genus(D_H) \ge 2$, 
then by Faltings' theorem there are only finitely many
$R \in D_H(L)$.
We can determine the genus of the curve $D_{H}$ 
using the monodromy data of $\varphi$ as we now
explain. Fix an embedding $L \subset \overline{L} \subset \C$.
Then $\varphi$ induces an \'{e}tale covering $C(\C) \setminus \varphi^{-1}(\BV(\varphi)) \rightarrow \PP^1(C) \setminus \BV(\varphi)$
of Riemann surfaces.
Since $\KK/L(\PP^1)$ is regular
(i.e. $\KK \cap \overline{L}=L$), we can identify
$G=\Gal(\KK/L(\PP^1))$ with the image of the monodromy representation
\cite[Chapter 4]{Miranda} of this covering.
Moreover, 
monodromy attaches (e.g.\ \cite[Corollary 4.10]{Miranda})
elements 
\begin{equation}\label{eqn:monodromy1}
\sigma_1,\sigma_2,\dotsc,\sigma_r \in G \setminus \{1\}
\end{equation}
to the branch values $\BV(\varphi)=\{\beta_1,\beta_2,\dotsc,\beta_r\}$,
satisfying 
\begin{equation}\label{eqn:monodromy2}
	G=\langle \sigma_1,\dotsc,\sigma_r \rangle,
	\qquad \sigma_1\sigma_2 \cdots \sigma_r=1.
\end{equation}
\begin{lem}\label{lem:monodromy}
$\displaystyle \genus(D_H)=1-[G:H]+\frac{1}{2} \sum_{i=1}^r \ind(\sigma_i,G/H)$.
\end{lem}
\begin{proof}
Since $L(D_H)=\KK^H$ we have $\deg(\pi_H)=[G:H]$.
We compute the genus of $D_H$ via the Riemann--Hurwitz
formula, and for this we need the ramification data
for $\pi_H$. This has been determined by
Lombardo, Lorenzo Garc\'{\i}a, 
Ritzenthaler and Sijsling \cite[Proposition 4.20]{Lombardo23}. 
Specifically, consider left multiplication by $\sigma_i$
on $G/H$; this induces a permutation of $G/H$ with
cycle type (say) $(e_1,e_2,\dotsc,e_k)$.
Then the fibre of $\pi_H$ above $\beta_i$
consists of $k$ points with multiplicities
$e_1,e_2,\dotsc,e_k$. Note that
\[
	\begin{split}
		\sum_{P \in \pi_H^{-1}(\beta_i)} e(P)-1 & =
		(e_1-1)+(e_2-1)+\cdots+(e_k-1) \\
		&=(e_1+\cdots+e_k)-k\\
		&=[G:H]-\# \Orb(\sigma_i,G/H)\\
		&=\ind(\sigma_i,G/H).
	\end{split}
\]
We now apply the Riemann--Hurwitz formula to $\pi_H$ to obtain
\[
		2\genus(D_H)-2 = -2 [G:H] + \sum_{i=1}^r \ind(\sigma_i,G/H).
\]
The lemma follows.
\end{proof}

\section{Proof of Theorem \ref{thm:primitive}}\label{sec:proof}
We note that $S_n$ has no proper
primitive subgroups $\ne A_n$ for $n \le 4$.
Thus we may suppose that $n \ge 5$.

Let $K$ be a number field and let $C/K$
be a curve of genus $g$. Let $n \ge 5$
and suppose $g>(n-1)^2$. 
Suppose the Jacobian $J$ is simple
or $J(K)$ is finite.
Suppose $C$
has infinitely many degree $n$
points with Galois group $A_n$ or $S_n$.
By Theorem~\ref{thm:single}, there is a degree $n$ morphism
$\varphi : C \rightarrow \PP^1$ such that all but
finitely many primitive degree $n$ divisors
are fibres $\varphi^*(\alpha)$ with $\alpha \in \PP^1(K)$.
Thus, to prove Theorem~\ref{thm:primitive},
we need to show that there are at most finitely many
$\alpha \in \PP^1(K) \setminus \BV(\varphi)$
such that the Galois group of the fibre $\varphi^*(\alpha)$
is primitive but not $A_n$, $S_n$.

\medskip

As in Section~\ref{sec:monodromy}, we write $\KK$ for the Galois closure of the function
field extension $K(C)/K(\PP^1)$ induced by $\varphi$, 
and we let $G^\prime=\Gal(\KK/K(\PP^1))$. As $\varphi$ has degree $n$,
we may identify $G^\prime$ as a subgroup of $S_n$.
\begin{lem}
$G^\prime=A_n$ or $S_n$.
\end{lem}
\begin{proof}
There infinitely many fibres $\varphi^*(\alpha)$ with 
$\alpha \in \PP^1(K) \setminus \BV(\varphi)$ that have Galois group $A_n$ or $S_n$.
Choose such an $\alpha$, let $P \in \varphi^*(\alpha)$; thus
the Galois group of the Galois closure of the extension
$K(P)/K$ has Galois group $S_n$ or $A_n$. Moreover, $P \in C$ is a degree $n$
point and we regard it as a degree $n$ place of $K(C)$. We let $\cP$ be a place of $\KK$
above $P$. By Lemma~\ref{lem:residue}, the extension $K(\cP)/K$
is Galois and its Galois group is 
isomorphic to the decomposition group $G^\prime_\cP \subseteq G^\prime$.
However, $K(P) \subseteq K(\cP)$ and the Galois group of the Galois closure
of $K(P)/K$ is either $A_n$ or $S_n$. In follows that $G^\prime=A_n$ or $S_n$.
\end{proof}

\begin{lem}\label{lem:BV}
Let $r=\# \BV(\varphi)$. Then $r \ge 2n+1$.
\end{lem}
\begin{proof}
Write $\BV(\varphi)=\{\beta_1,\dotsc,\beta_r\}$.
We make use of the Riemann--Hurwitz formula applied to $\varphi$. Thus
\[
	2g-2 \; = \; -2n + \sum_{i=1}^r \sum_{P \in \varphi^*(\beta_i)} e(P)-1
	\; \le \; -2n + r (n-1).
\]
However, $g \ge (n-1)^2+1$. Putting these together gives,
\[
	r \; \ge \; 2(n-1) + \frac{2n}{n-1} \; > \; 2n. 
\]
\end{proof}
As in Section~\ref{sec:monodromy}, let $L=\KK \cap \overline{K}$,
and recall that $L/K$ is a finite Galois extension.
Let $G=\Gal(\KK/L(\PP^1))$.
\begin{lem}\label{lem:threecases}
	\begin{enumerate}[(a)]
		\item Suppose $G^\prime = A_n$. 
	Then $L=K$ and $G=A_n$.
		\item Suppose $G^\prime = S_n$.
		Then
			\begin{enumerate}[(i)]
		\item either $L=K$ and $G = S_n$,
		\item or $L/K$ is quadratic and $G=A_n$.
		\end{enumerate}
	\end{enumerate}
\end{lem}
\begin{proof}
As $\varphi$ is ramified, the extension $\KK/L(\PP^1)$
is non-trivial. Therefore, by the exactness of \eqref{eqn:exact},
the group $G$
is a non-trivial normal subgroup of $G^\prime$.
As $n \ge 5$, the only non-trivial normal subgroup of $A_n$ is $A_n$,
and the only non-trivial normal subgroups of $S_n$ are $S_n$ and $A_n$.
The lemma follows.
\end{proof}

\begin{lem}\label{lem:genus}
Let $H$ be a transitive subgroup of $G$.
Suppose the following.
\begin{enumerate}[(a)]
\item If $G=A_n$ then $H$ is maximal in $G$.
\item If $G=S_n$ then either $H \ne A_n$ and maximal in $S_n$,
or $H \subset A_n$ and maximal in $A_n$.
\end{enumerate}
Then $\genus(D_H) \ge 2$.
\end{lem}
\begin{proof}
We make use of the formula in Lemma~\ref{lem:monodromy}.
Note that the $\sigma_i$ appearing in the formula
are all non-trivial elements on $G$ (their
properties are summarized in \eqref{eqn:monodromy1}
and \eqref{eqn:monodromy2}). In particular,
$\ind(\sigma_i,G/H)\ge \ind(G,G/H)$. We make
use of the lower bounds for $\ind(G,G/H)$ in Lemma~\ref{lem:ind}.
Write 
\[
\rho(G,H) \; = \; \frac{\ind(G,G/H)}{[G:H]}.
\]
 Then, by Lemma~\ref{lem:monodromy}
\[
\begin{split}
\genus(D_H) & \ge 1-[G:H]+\frac{r}{2} \ind(G,G/H)\\
& \ge 1-[G:H]+\frac{2n+1}{2} \ind(G,G/H) \qquad \text{(by Lemma~\ref{lem:BV})}\\
&=1+\left(\frac{(2n+1)\rho(G,H)}{2} -1 \right) [G:H].
\end{split}
\]
Hence $\genus(D_H) \ge 2$ provided
\begin{equation}\label{eqn:criterion}
\rho(G,H)>\frac{2}{2n+1}.
\end{equation}

Suppose first $G=A_n$ and $H$ is maximal in $G$.
By Lemma~\ref{lem:ind}, we have $\rho(A_n,H) \ge 1/4$,
and so \eqref{eqn:criterion} holds as $n \ge 5$.

Next suppose $G=S_n$ and $H \ne A_n$ is maximal in $S_n$.
Then $\rho(S_n,H) \ge 1/6$ by Lemma~\ref{lem:ind}, and so
\eqref{eqn:criterion} holds for $n\ge 6$. We return to the case
$n=5$ later.

Next suppose $G=S_n$ and $H$ is maximal in $A_n$.
Then $\rho(S_n,H) \ge 1/8$ by Lemma~\ref{lem:ind},
and so \eqref{eqn:criterion} holds for $n \ge 8$.

It remains to consider the cases $n=5$, $6$, $7$.
Using the computer algebra package \texttt{Magma} \cite{magma}
we enumerated (up to conjugacy) all transitive subgroups $H$
that are either maximal in $S_n$ but $\ne A_n$ or maximal in $A_n$,
and for each computed the ratio $\rho(S_n,H)$. The results
are given in Table~\ref{table:rho}. The fact that the last column
is positive shows that \eqref{eqn:criterion} holds in all cases and completes the proof.
\begin{table}
\begin{tabular}{|c|c|c|c|c|c|c|}
\hline
$n$ & $H$  & $\# H$ & $[S_n:H]$ & $\ind(S_n,S_n/H)$ & $\rho(S_n,H)$ & $\rho(S_n,H)-2/(2n+1)$ \\
\hline
 $5$ & $D_5$ & $10$ & $12$ & $4$ & $1/3$ & $5/33 \approx 0.15$\\ 
\hline
$5$ & $F_5$ & $20$ & $6$ &  $2$ & $1/3$ & $5/33 \approx 0.15$\\
\hline
$6$ & $S_4$ & $24$ & $30$ & $12$ & $2/5$ & $16/65 \approx 0.25$\\ 
\hline
 $6$ & $C_3^2 : C_4$ & $36$ & $20$ & $8$ & $2/5$ & $16/65 \approx 0.25$ \\
\hline
 $6$ & $A_5$ & $60$ & $12$ & $4$ & $1/3$ & $7/39 \approx 0.18$ \\
\hline
$6$ & $C_2 \times S_4$ & $48$ & $15$ & $4$ & $4/15$ & $22/195 \approx 0.11$ \\
\hline
 $6$ & $S_3 \wr C_2$ & $72$ & $10$ & $3$ & $3/10$ & $19/130 \approx 0.15$\\
\hline
$6$ & $S_5$ &  $120$ & $6$ & $1$ & $1/6$ & $1/78 \approx 0.01$\\
\hline
$7$ & $F_7$ & $42$ &  $120$ & $56$ & $7/15$ & $1/3 \approx 0.33$\\
\hline
$7$ & $\PSL(2,7)$ & $168$ & $30$ & $12$ & $2/5$ & $4/15 \approx 0.27$\\
\hline
\end{tabular}
\caption{The table gives the ratio
$\rho(S_n,H)=\ind(S_n,S_n/H)/[S_n:H]$, for $n=5$, $6$, $7$, and for
conjugacy classes of transitive subgroups $H \subset S_n$,
which are either maximal in $A_n$, or maximal in $S_n$ but $\ne A_n$.
Here, in the second column, we give the name of the group $H$
according to the Group Names Database \texttt{GroupNames.org}.}
\label{table:rho}
\end{table}
\end{proof}

We now complete the proof of Theorem~\ref{thm:primitive}.
As observed at the beginning of the section, 
we need to show that there are at most finitely many
$\alpha \in \PP^1(K) \setminus \BV(\varphi)$
such that the Galois group of the fibre $\varphi^*(\alpha)$
is primitive but not $A_n$, $S_n$.
It is therefore enough, for each primitive subgroup $I \subset G^\prime$,
with $I \ne A_n$, $S_n$, to show that there are finitely many
$\alpha \in \PP^1(K) \setminus \BV(\varphi)$ such that the
fibre $\varphi^*(\alpha)$ has Galois group $I$.
Fix a primitive subgroup $I \subset G^\prime$, $I \ne A_n$, $S_n$,
and suppose there are infinitely many $\alpha \in \PP^1(K) \setminus \BV(\varphi)$
such that the fibre $\varphi^*(\alpha)$ has Galois group $I$.
Since $C(L)$ is finite by Faltings' theorem, only finitely
many of these fibres contain a point of $C(L)$. 
Thus for infinitely many of the fibres there is a primitive
degree $n$ point $P \in C \setminus C(L)$ whose Galois group is $I$.
There are finitely many possibilities for the groups $J$, $H$
in Lemma~\ref{lem:IJ}. As $I \ne A_n$, $S_n$,
by the Lemma, $H \ne A_n$, $S_n$ is a transitive subgroup of $G$.
If $H \subset A_n$ then let $H^\prime$ be a maximal subgroup of $A_n$
containing $H$. If $H \not \subset A_n$, then in this case $G=S_n$,
and we let $H^\prime$ be a maximal subgroup of $S_n$ containing $H$,
and note that $H^\prime \ne A_n$. As $H \subseteq H^\prime$
we have $\KK^{H^\prime} \subseteq \KK^H$ and so
the map $\pi_H : D_H \rightarrow \PP^1$ factors
via the map $\pi_{H^\prime} : D_H^\prime \rightarrow \PP^1$.
In particular, by Lemma~\ref{lem:genus} (applied with $H^\prime$
in place of $H$)
the curve $D_{H^\prime}$ has genus $\ge 2$
and therefore so does the curve $D_H$.
However, $D_H(L)$ is finite. This gives a contradiction
and completes the proof.

\bibliographystyle{abbrv}
\bibliography{Primitive}
\end{document}